\documentclass[10pt, a4paper]{amsart}
\usepackage[british]{babel}
\usepackage{amsmath}
\usepackage{amssymb}
\usepackage{amsthm}
\usepackage{graphicx}
\usepackage{url}
\usepackage{enumerate}
\usepackage{xcolor}
\usepackage{mathrsfs} 
\usepackage{t1enc}
\usepackage{mathtools}
\usepackage{placeins} 
\usepackage{tikz}
\usepackage{textcomp}
\usepackage{appendix}
\usetikzlibrary{matrix}
\usetikzlibrary{arrows,automata}
\usetikzlibrary{decorations.pathreplacing}

\newtheorem{thm}{Theorem}[section]

\newtheorem{theorem}{Theorem}[section]
\newtheorem{lemma}[thm]{Lemma}
\newtheorem{proposition}[thm]{Proposition}
\newtheorem{observation}[thm]{Observation}
\newtheorem{corollary}[thm]{Corollary}

\newtheorem{definition}[thm]{Definition}
\newtheorem{remark}[thm]{Remark}
\newtheorem{example}[thm]{Example}

\DeclareMathOperator{\diam}{diam}
\DeclareMathOperator{\orb}{Orb}

\newcommand{\eps}{\varepsilon}

\newcommand{\N}{\mathbb{N}}

\newcommand{\f}{f\colon [0,1]\to[0,2]}

\newcommand{\Fix}{\mathrm{Fix}}

\begin{document}

\title{On zero entropy homeomorphisms of the pseudo-arc}

\author{Jernej \v{C}in\v c}

\subjclass[2010]{37E05, 37B45}
\keywords{pseudo-arc, inverse limit, topological entropy}
\begin{abstract}
	In this paper we study interval maps $f$ with zero topological entropy that are crooked; i.e. whose inverse limit with $f$ as the single bonding map is the pseudo-arc. We show that there are uncountably many pairwise non-conjugate zero entropy crooked interval maps with different sets of fixed points.
	We also show that there are uncountably many zero entropy crooked maps that are pairwise non-conjugate and have exactly two fixed points. Furthermore, we provide a characterization of crooked interval maps that are under or above the identity diagonal.
\end{abstract}
\maketitle

\section{Introduction}
The pseudo-arc is a one-dimensional compact connected metric space (a \emph{continuum}) with extraordinary properties. It is homemorphic to each of its non-degenerate subcontinua and for any two points in the pseudo-arc there is a homemorphism of the pseudo-arc taking one to the other yet it contains no arcs. 
	Since it is so intrinsically complicated it was initially thought
that the topological complications would restrict the type of dynamics
	it could support. However, it was recently shown by Boro\'nski, Oprocha and the author \cite{BCO} that for any non-negative real value $r$ there is a homeomorphism of the pseudo-arc with topological entropy $r$. Nonetheless, the question about dynamical richness of pseudo-arc homeomorphisms with a specific value of topological entropy remains to be explored.

The aim of this paper is to study dynamical variety of maps with  topological entropy $0$ that give as single bonding maps the pseudo-arc in the inverse limit. Henderson in \cite{He} provided a construction to obtain such a map, we call it here the Henderson map $f_H$, see Example~\ref{ex:Hen} for more details. However, it is a priori unclear if the construction given in \cite{He} can be modified in order to obtain interval maps non-conjugate to $f_H$ whose inverse limit is the pseudo-arc and have topological entropy $0$. This question has been recently circulating in the community working on the interactions between Continuum Theory and Topological Dynamical Systems; in the paper of Drwięga and Oprocha \cite{DO} (see the end of page 197) the authors mention that there are no other known techniques to construct such maps and thus implicitly pose this question.

We answer the question in the affirmative and in particular show that there are uncountably many such pairwise non-conjugate interval maps with different sets of fixed points as well as with exactly two fixed points. Moreover, we provide a characterization of crooked interval maps that are under the identity diagonal.
Our characterization of crooked under diagonal maps says it is enough to require crookedness in an arbitrary small neighbourhood $U$ of point $1$ and in addition to find a point $x\in U$ that is attracted to $0$  by this map. Analogous results also hold for maps that are above the identity diagonal.

In a sense, this paper also initiates the study that is complementary to the study in \cite{CO}. There, Oprocha and the author prove that generic maps in the closure of the set of all maps with dense set of periodic points as well as in the set of all Lebesgue measure-preserving interval maps as single bonding maps give the pseudo-arc in the inverse limit. In that context, the topological entropy of maps is known to be $\infty$ (see e.g. Proposition 26 in \cite{BT} and Remark in \cite{BCOT}).  Mouron \cite{Mouron} has proved that crooked maps can have topological entropy $\infty$ or $0$. 

 Our results give novel implications also from the following perspective from Surface Dynamics. Barge and Martin \cite{BM} provided a technique to embed the inverse limit space of any interval map as the single bonding map as the attractor of some planar homeomorphism. Now take any map $f\in \mathcal T$ from Theorem~\ref{thm:non-conjstrobedient}; since $f$ is crooked, the attractor using the technique from \cite{BM} will be the pseudo-arc. But since any two maps from $\mathcal{T}$ are not-conjugate, it follows that the two induced embeddings of the pseudo-arc are dynamically non-equivalent (see Definition 4.1 from \cite{CO}).	Thus we argued that Theorem~\ref{thm:non-conjstrobedient} gives uncountably many  dynamically non-conjugate embeddings of the pseudo-arc in the plane with planar homeomorphisms having topological entropy $0$.

 Furthermore, our study can be seen as the continuation of the study towards understanding the interaction between the dynamics of  bonding maps and topological complexity of their inverse limits. More precisely, the question that propelled a significant amount of research in the last three decades was to characterize the conditions on factor spaces and bonding maps to ensure that their corresponding inverse limit spaces contain indecomposable subcontinua.
  One direction of this line of study culminated in a very nice paper of Darji and Kato \cite{DK} where the authors, using recent developments in local entropy theory, show that inverse limit of a graph map with positive topological entropy contains an indecomposable subcontinuum. The other direction of this line of research has, however, not been explored to a great extent and our research can be thought as the first step in that direction. Our paper can also be treated as a continuation of the work of Greenwood and \v Stimac \cite{GS} where the authors characterized interval maps that give arcs in the inverse limit as single bonding maps.
 
 Let us give a brief outline of the paper. In Section~\ref{sec:preliminaries} we provide preliminary definitions used throughout the paper. In Section~\ref{sec:obedient} we define the collection of maps that we will work with in the rest of the paper, provide initial examples that motivate our study and explain our approach. 
  Finally, in Section~\ref{sec:non-conjugate} we prove the main results about non-conjugate under and above identity diagonal interval maps with zero topological entropy that give pseudo-arc in the inverse limit as well as provide their characterization.

\section{General preliminaries}\label{sec:preliminaries}

We denote the set of {\em natural numbers}  by $\N:=\{1,2,3,\ldots\}$, and let $\N_0=\N\cup\{0\}$.
By $I$ let us denote the {\em unit interval} $[0,1]\subset \mathbb{R}$. A {\em map} $f\colon I\to I$ is a continuous surjective function. Denote the {\em set of interval maps} by $C(I)$. We equip $C(I)$ with the metric of uniform convergence. Let $f^n=f\circ\ldots\circ f$ denote the {\em $n$-th iterate} of $f$. 
We call a point $c\in (0,1)$ a {\em critical point} of $f\colon I\to I$ if for every open $J\ni c$, $f|_J$ is not one-to-one. In particular, we say that $c$ is a {\em strict critical point} if in addition there exists an interval $c\in (t_1,t_2)\subset J$ such that $f|_{(t_1,c)}$ and $f|_{(c,t_2)}$ are one-to-one.
An interval map $f:I\to I$ is called {\em piecewise monotone}, if $f$ has finitely many critical points. Given a map $f\colon I\to I$, the {\em orbit of $x\in I$ under $f$} is the set $$\orb(x,f):=\{f^n(x): n\geq 0\}.$$

We say that a point $x\in I$ is {\em $f$-attracted} to $y\in I$ if $f^n(x)\to y$ as $n\to\infty$. 
The {\em set of fixed points} of $f$ will be denoted by 
$$\Fix(f):=\{y\in I: f(y)=y\}.$$ 
We say that maps $f,g\in C(I)$ are {\em conjugate}, if there exists a homeomorphism $h:I\to I$ such that $g=h\circ f\circ h^{-1}$.

A {\em continuum} is a non-empty, compact, connected, metric space. In the definition of inverse limit space we can assume without loss of generality that the bonding functions are surjective.
Let $\{I_i\}_{i\in\N}$ be a collection of closed intervals from $I$.   An {\em inverse sequence of $\{I_i\}_{i\in\N}$}  is a sequence $(I_i,f_i)_{i\in \N}$ where  $f_i\colon I_{i+1}\to I_{i}$, $i\in\N$.
 The {\em inverse limit} is defined by:
$$\underleftarrow{\lim}(I_i, f_i)=\{(x_1, x_2, \ldots): f_i(x_{i+1})=x_{i}, i\in\N\}\subset  I_1\times I_2\times\ldots.$$
Equipped with the {\em product topology}, $\underleftarrow{\lim}\{I_i, f_i\}$ is a continuum, and it is {\em arc-like}, i.e. for every $\eps>0$ there exists an $\eps$-mapping onto an interval. For $\eps>0$, a map $g$ from $X$ onto $Y$ is an {\em $\eps$-mapping} if for every $y\in Y$, the diameter of $g^{-1}(y)$ is less than $\eps$.
The {\em coordinate projections} are defined 
by $\pi_i\colon \underleftarrow{\lim}(I_i, f_i)\to I$, $\pi_i((x_1, x_2, \ldots))=x_i$, for $i\in\N$ and they are all continuous.

If $f_i=f$ and $I_i=I$ for all $i\in\N$, we denote the inverse limit by $\hat I:=\underleftarrow{\lim}(I, f)$. Then there exists a 
homeomorphism $\hat f\colon \hat I\to \hat I$, given by $$\hat f((x_0, x_1, \ldots)):=(f(x_0), f(x_1), f(x_2), \ldots)=(f(x_0), x_0, x_1, \ldots).$$
It is called the {\em natural extension} of $f$ (or sometimes {\em shift homeomorphism}).

\section{Under diagonal, above diagonal maps and preliminary examples}\label{sec:obedient}
\begin{definition}\label{def:obedient}
			{\upshape
	 A map $f:I\to I$ is called \emph{under diagonal} if $f(x)\leq x$ for every $x\in [0,1]$. 
A map $f:I\to I$ is called \emph{$0$-attracted} if $f(0)=0$, $f(1)=1$ and $f(x)<x$ for every $x\in (0,1)$. Denote the set of under diagonal interval maps by $C_{ud}(I)$ and the set of $0$-attracted interval maps by $C_{0}(I)\subset C_{ud}(I)$.
We also define \emph{above diagonal} interval maps $f$ by $f(x)\geq x$ for every $x\in [0,1]$.  A map is called {\em $1$-attracted} if additionally $f(x)>x$ for every $x\in (0,1)$. We denote the set of above diagonal interval maps by $C_{ad}(I)$ and the set of $1$-attracted maps by $C_{1}(I) \subset C_{ud}(I)$. }
\end{definition}

The following theorem shows that the inverse limits with piecewise monotone under diagonal maps are topologically unique. This follows from Theorem B8 from \cite{AC} since by the definition of under diagonal maps, for every $x\in I$ there exists $p\in \Fix(f)$ such that $f^{n}(x) \to p$ as $n\to \infty$ (in the terminology of \cite{AC} every $x\in I$ is $f$-attracted to some $p\in \Fix(f)$).

\begin{theorem}\label{thm:pwnarc}
	If $f:I\to I$ is under diagonal and piecewise monotone then $\hat{I}_f$ is homeomorphic to an arc.
\end{theorem}

\begin{definition}\label{def:crooked}
		{\upshape
	Let $f\in C(I)$ and let $a,b\in I$ and let $\delta>0$. We say that $f$ is {\em $\delta$-crooked between $a$ and $b$} if for every two points $c, d \in I$ such that $f(c) = a$ and $f(d) = b$, there is a point $c'$ between $c$ and $d$ and there is a point $d'$ between $c'$ and $d$ such that $|b - f(c')| < \delta$ and $|a - f(d')| < \delta$. We will say that $f$ is {\em $\delta$-crooked} if it is $\delta$-crooked between every pair of points from $I$.  
	We will simply say that $f$ is {\em crooked}, if for every $\delta>0$ there exists $n\in \N$ such that $f^n$ is $\delta$-crooked.}
\end{definition}

From \cite[Theorem 3.5]{MK} we have the following characterization.

\begin{proposition}\label{prop:charpseudo-arc}
	Space $\hat I_f$ is the pseudo-arc if and only if $f$ is crooked.
\end{proposition}

 In particular, Theorem~\ref{thm:pwnarc} says that crooked interval maps cannot be piecewise monotone.

Next examples show that even for $0$-attracted maps there is no connection between interval maps with infinitely many strict critical points and crookedness. With them we want to emphasize the topological variety of $\hat I_f$ despite simple dynamics of $f$.

\begin{example}
			{\upshape
There exists a $0$-attracted map $f\in C(I)$ with infinitely many strict critical points so that $\hat I_f$ is homeomorphic to an arc,  see Figure~\ref{fig:arc}. Indeed, this follows from an extension of Lemma B.5 in \cite{AC} since every $x\in [0,1)$ is $f$-attracted to $0$, see Remark B.7 in \cite{AC}. }
\end{example}

\begin{figure}
	\begin{tikzpicture}[scale=4]
	\draw(0,0)--(0,1)--(1,1)--(1,0)--(0,0);
	\draw[dashed] (0,0)--(1,1);
	\draw[thick] (0,0)--(0.3,0.05)--(0.32,0.3)--(0.34,0.1)--(0.36,0.3)--(0.37,0.18)--(0.38,0.22);
	\draw[thick] (1,1)--(0.5,0.35)--(0.48,0.05)--(0.46,0.31)--(0.44,0.1)--(0.43,0.25)--(0.42,0.18);
	\node[circle,fill, inner sep=1] at (0.4,0.2){};
	\end{tikzpicture}
	\caption{Graph of a $0$-attracted interval map $f$ with infinitely many critical points for which the inverse limit is an arc.}
	\label{fig:arc}
\end{figure}

The following example shows that $0$-attracted ($1$-attracted) maps can have inverse limits that are neither pseudo-arc nor arc.

A {\em double $\sin(1/x)$ continuum} is a continuum which consisting of two rays (homeomorphic image of $[0,1)$) that both accumulate on a unique arc, see the right picture on Figure~\ref{fig:sin1/x}.
\begin{example}\label{ex:doublesin(x)}
			{\upshape
	There exists a $1$-attracted interval map $f$ with infinitely many critical points so that $\hat I_f$ is not homeomorphic to an interval and $\hat I_f$ is a union of countably infinitely many double $\sin(1/x)$ continua, see Figure~\ref{fig:sin1/x}. An important feature of map $f$ is that the absolute value of the slope of map restricted to intervals $[a_1,b_1]$ is $1$ and for every $i\in \N$ there is a subinterval of $[a_{i+1},b_{i+1}]$ that maps onto $[a_i,b_i]$ and that $\diam(f^{i-1}([a_i,b_i]))=\diam([f(a_1),f(b_1)])$ (compare this with Example 4.12 in \cite{AC}). The subcontinuum $H:=\underleftarrow{\lim}([a_i,b_i],f|_{[a_i,b_i]})$ is homeomorphic to a double $\sin(1/x)$ continuum (see the right picture on Figure~\ref{fig:sin1/x}). Also ${\hat f}^{n}(H)$ will be double $\sin(1/x)$ continuum for every $n\in \N$.}
\end{example}

  Requiring in Example~\ref{ex:doublesin(x)} that $\diam(f^{i-1}([a_i,b_i]))\to 0$ as $i\to \infty$ would yield that $\hat I_f$ is an arc (see Example 4.12 in \cite{AC}).

\begin{figure}
	
	\begin{tikzpicture}[scale=6.3]
	\vspace{-20pt}
	\draw[thick] (1,1)--(8/10,1)--(0.78,0.92)--(0.7,0.99)--(0.68,0.9)--(0.6,0.8)--(0.58,0.705)--(0.52,0.78)--(0.5,0.68)--(0.43,0.6)--(0.405,0.525)--(0.375,0.575)--(0.35,0.5)--(0.29,0.43)--(0.262,0.365)--(0.24,0.41)--(0.22,0.35)--(0.18,0.29)--(0.161,0.245)--(0.153,0.27)--(0.14,0.22)--(0.115,0.18);
	\draw[dashed] (0.9,0.9)--(0.8,0.9);
	\draw[dashed](0.14,0.14)--(0.14,0.22)--
	(0.22,0.22)--(0.22,0.35)--(0.35,0.35)--(0.35,0.5)--(0.5,0.5)--(0.5,0.68)--(0.68,0.68)--(0.68,0.9)--(0.9,0.9)--(0.9,0);
	\draw[dashed](0.115,0.115)--
	(0.115,0.18)--(0.18,0.18)--(0.18,0.29)--(0.29,0.29)--(0.29,0.43)--(0.43,0.43)--(0.43,0.6)--(0.6,0.6)--(0.6,0.8)--(0.8,0.8)--(0.8,1);
	\draw[dashed] (0,0)--(1,1);
	\draw (0,0) -- (0,1) -- (1,1) -- (1,0) -- (0,0);
	\draw[dashed] (0.8,0)--(0.8,0.8); 
	\draw[dashed] (0.68,0)--(0.68,0.68); 
	\draw[dashed] (0.6,0)--(0.6,0.6); 
	\draw[dashed] (0.5,0)--(0.5,0.5); 
	\draw[dashed] (0.43,0)--(0.43,0.43); 
	\draw[dashed] (0.35,0)--(0.35,0.35); 
	\draw[dashed] (0.29,0)--(0.29,0.29); 
	\draw[dashed] (0.22,0)--(0.22,0.22); 
	\draw[dashed] (0.18,0)--(0.18,0.18); 
	\draw[dashed] (0.14,0)--(0.14,0.14); 
	\draw[dashed] (0.115,0)--(0.115,0.115);

	\draw[dashed] (0,0.9)--(0.9,0.9); 
	\draw[dashed] (0,0.8)--(0.8,0.8); 
	\draw[dashed] (0,0.68)--(0.68,0.68); 
	\draw[dashed] (0,0.6)--(0.6,0.6); 
	\draw[dashed] (0,0.5)--(0.5,0.5); 
	\draw[dashed] (0,0.43)--(0.43,0.43); 
	\draw[dashed] (0,0.35)--(0.35,0.35); 
	\draw[dashed] (0,0.29)--(0.29,0.29); 
	\draw[dashed] (0,0.22)--(0.22,0.22); 
	\draw[dashed] (0,0.18)--(0.18,0.18);

	\draw[dashed] (0.68,0.9)--(0.68,1); 
	\draw[dashed] (0.5,0.68)--(0.5,0.8); 
	\draw[dashed] (0.35,0.5)--(0.35,0.6); 
	\draw[dashed] (0.22,0.35)--(0.22,0.43); 
	\draw[dashed] (0.14,0.22)--(0.14,0.29);
	
	\node[circle,fill, inner sep=1] at (1,1){};
	\node[circle,fill, inner sep=1] at (0.9,0.9){};
	\node[circle,fill, inner sep=1] at (0.8,0.8){};
	\node[circle,fill, inner sep=1] at (0.77,0.77){};
	\node[circle,fill, inner sep=1] at (0.71,0.71){};
	\node[circle,fill, inner sep=1] at (0.68,0.68){};
	\node[circle,fill, inner sep=1] at (0.6,0.6){};
	\node[circle,fill, inner sep=1] at (0.57,0.57){};
	\node[circle,fill, inner sep=1] at (0.53,0.53){};
	\node[circle,fill, inner sep=1] at (0.5,0.5){};
	\node[circle,fill, inner sep=1] at (0.43,0.43){};
	\node[circle,fill, inner sep=1] at (0.41,0.41){};
	\node[circle,fill, inner sep=1] at (0.375,0.375){};
	\node[circle,fill, inner sep=1] at (0.35,0.35){};
	\node[circle,fill, inner sep=1] at (0.29,0.29){};
	\node[circle,fill, inner sep=1] at (0.27,0.27){};
	\node[circle,fill, inner sep=1] at (0.24,0.24){};
	\node[circle,fill, inner sep=1] at (0.22,0.22){};
	\node[circle,fill, inner sep=1] at (0.18,0.18){};
	\node[circle,fill, inner sep=1] at (0.14,0.14){};
	\node[circle,fill, inner sep=1] at (0.115,0.115){};
	\node[circle,fill, inner sep=1] at (1,0){};
	\node[circle,fill, inner sep=1] at (0.99,0){};
	\node[circle,fill, inner sep=1] at (0.91,0){};
	\node[circle,fill, inner sep=1] at (0.9,0){};
	\node[circle,fill, inner sep=1] at (0.8,0){};
	\node[circle,fill, inner sep=1] at (0.77,0){};
	\node[circle,fill, inner sep=1] at (0.71,0){};
	\node[circle,fill, inner sep=1] at (0.68,0){};
	\node[circle,fill, inner sep=1] at (0.6,0){};
	\node[circle,fill, inner sep=1] at (0.5,0){};
	\node[circle,fill, inner sep=1] at (0.43,0){};
	\node[circle,fill, inner sep=1] at (0.35,0){};
	\node[circle,fill, inner sep=1] at (0.29,0.){};
	\node[circle,fill, inner sep=1] at (0.22,0){};
	\node[circle,fill, inner sep=1] at (0.18,0){};
	\node[circle,fill, inner sep=1] at (0.14,0){};
	\node[circle,fill, inner sep=1] at (0.115,0){};
	\node[circle,fill, inner sep=1] at (0.57,0){};
	\node[circle,fill, inner sep=1] at (0.53,0){};
	\node[circle,fill, inner sep=1] at (0.41,0){};
	\node[circle,fill, inner sep=1] at (0.375,0.){};
	\node[circle,fill, inner sep=1] at (0.27,0){};
	\node[circle,fill, inner sep=1] at (0.24,0){};
	\node[circle,fill, inner sep=1] at (0.165,0){};
	\node[circle,fill, inner sep=1] at (0.155,0){};
	\node at (-0.05,1) {\scriptsize $1$};
	\node at (0,-0.03) {\scriptsize $0$};
	\node at (0.15,-0.03) {\scriptsize $a'_5$};
	\node at (0.19,-0.03) {\scriptsize $b'_5$};
	\node at (0.23,-0.03) {\scriptsize $a'_{4}$};
	\node at (0.3,-0.03) {\scriptsize $b'_4$};
	\node at (0.36,-0.03) {\scriptsize $a'_3$};
	\node at (0.44,-0.03) {\scriptsize $b'_{3}$};
	\node at (0.5,-0.03) {\scriptsize $a'_2$};
	\node at (0.547,-0.035) {\scriptsize $a_{2}$};
	\node at (0.58,-0.03) {\scriptsize $b_{2}$};
	\node at (0.62,-0.03) {\scriptsize $b'_{2}$};
	\node at (0.68,-0.03) {\scriptsize $a'_{1}$};
	\node at (0.715,-0.035) {\scriptsize $a_{1}$};
	\node at (0.775,-0.03) {\scriptsize $b_{1}$};
	\node at (0.81,-0.03) {\scriptsize $b'_{1}$};
	\node at (0.9,-0.03) {\scriptsize $a'_{0}$};
	\node at (0.94,-0.03) {\scriptsize $b_{0}$};
	\node at (0.985,-0.035) {\scriptsize $a_{0}$};
	\node at (1.02,-0.03) {\scriptsize $b'_{0}$};
	\end{tikzpicture}
	\hspace{20pt}
	\raisebox{2cm}{
		\begin{tikzpicture}[scale=5]
		\draw (0.1,0.75)--(0.6,0.75);
		\draw (0.4,0.62)--(0.6,0.62);
		\draw (0.4,0.56)--(0.6,0.56);
		\draw (0.4,0.53)--(0.6,0.53);
		\draw (0.4,0.51)--(0.6,0.51);
		\draw[domain=270:450] plot ({0.6+0.065*cos(\x)}, {0.685+0.065*sin(\x)});
		\draw[domain=90:270] plot ({0.4+0.03*cos(\x)}, {0.59+0.03*sin(\x)});
		\draw[domain=270:450] plot ({0.6+0.015*cos(\x)}, {0.545+0.015*sin(\x)});
		\draw[domain=90:270] plot ({0.4+0.01*cos(\x)}, {0.52+0.01*sin(\x)});
		\draw[thick] (0.4,0.47)--(0.6,0.47);
		\node[circle,fill, inner sep=1] at (0.4,0.47){};
		\node[circle,fill, inner sep=1] at (0.6,0.47){};
		\begin{scope}[yscale=-1,xscale=-1,yshift=-0.94cm,xshift=-1cm]
		\draw (0.1,0.75)--(0.6,0.75);
		\draw (0.4,0.62)--(0.6,0.62);
		\draw (0.4,0.56)--(0.6,0.56);
		\draw (0.4,0.53)--(0.6,0.53);
		\draw (0.4,0.51)--(0.6,0.51);
		\draw[domain=270:450] plot ({0.6+0.065*cos(\x)}, {0.685+0.065*sin(\x)});
		\draw[domain=90:270] plot ({0.4+0.03*cos(\x)}, {0.59+0.03*sin(\x)});
		\draw[domain=270:450] plot ({0.6+0.015*cos(\x)}, {0.545+0.015*sin(\x)});
		\draw[domain=90:270] plot ({0.4+0.01*cos(\x)}, {0.52+0.01*sin(\x)});
		\end{scope}
		\end{tikzpicture}}
	\caption{Left: graph of an example of a  $1$-attracted interval map $f$ for which the inverse limit has (double) $\sin(1/x)$-continua. Right: double $\sin(1/x)$-continuum.}
	\label{fig:sin1/x}
\end{figure}

\begin{example}\label{ex:Hen}
	{\upshape
		Henderson's map $f_H:I\to I$ defined in \cite{He} is an example of a crooked $0$-attracted map $f\in C(I)$ with infinitely many strict critical points. 
		Its construction can be roughly described as starting with a map $f(x)=x^2$ and perturbing its map with infinitely many $v$-shaped notches which accumulate in a particular fashion on the point $(1,1)\subset I\times I$ (see Figure~\ref{fig:Hen}).}
\end{example}
\begin{figure}
	\begin{tikzpicture}[scale=4]
	\draw(0,0)--(0,1)--(1,1)--(1,0)--(0,0);
	\draw[dashed] (0,0)--(1,1);
	\draw[domain=0:3/4, thick, smooth, variable=\x] plot ({\x},{\x*\x});
	\draw[thick](3/4,9/16)--(0.77,0.6)--(0.8,0.55)--(5/6,25/36);
	\draw[domain=5/6:0.86, thick, smooth, variable=\x] plot ({\x},{\x*\x});
	\node[circle,fill, inner sep=0.8] at (0.72,0.5184){};
	\node[circle,fill, inner sep=0.8] at (0.77,0.6){};
	\node[circle,fill, inner sep=0.8] at (0.8,0.55){};
	\node[circle,fill, inner sep=0.8] at (5/6,25/36){};
	\node[circle,fill, inner sep=0.8] at (0.935,0.87){};
	\node[circle,fill, inner sep=0.8] at (0.96,0.91){};
	\node[circle,fill, inner sep=0.8] at (0.981,0.95){};
	\draw[thick](0.86,0.7396)--(0.875,0.71)--(8/9,64/81);
	\draw[domain=8/9:0.92, thick, smooth, variable=\x] plot ({\x},{\x*\x});
	\node[circle,fill, inner sep=0.8] at (0.86,0.7396){};
	\node[circle,fill, inner sep=0.8] at (0.875,0.71){};
	\node[circle,fill, inner sep=0.8] at (8/9,64/81){};
	\end{tikzpicture}
	\caption{Graph of the map $f_{H}$ from \cite{He}.}
	\label{fig:Hen}
\end{figure}

\section{Pairwise non-conjugate crooked maps derived from $f_H$}\label{sec:non-conjugate}

Note that crookedness is preserved under topological conjugacies, i.e., provided $f\in C(I)$ is crooked and $g\in C(I)$  is conjugate to $f$, then $g$ is also crooked. This indeed follows immediately from Proposition~\ref{prop:charpseudo-arc} and the fact that if $f$ is conjugate to $g$, then $\hat I_f$ is homeomorphic to $\hat I_g$. First let us state a standard observation.

\begin{observation}\label{obs:subcontinua}
			{\upshape
	Let $H$ be a subcontinuum of $\hat I_f$. Then there exists an inverse sequence $(J_i,f_i)_{i\in \N}$ where $J_i\subset I$ are closed intervals or points and $f_i=f|_{J_i}$ continuous surjections so that $H=\underleftarrow{\lim}(J_i, f_i)$. In particular, if $H$ is non-degenerate then $J_i$ are non-degenerate for all but finitely many $i\in \N$. }
\end{observation}

The following proposition was proven in \cite{Bi1,Bi2}. We will use it when studying the subcontinua of $\hat I_f$.

For $i> j\geq 1$ let $f_{i,j}:=f_{j}\circ\ldots \circ f_{i-1}$. For completeness let $f_{i,i}:=\mathrm{id}$.

\begin{proposition}\cite[Proposition 3.2]{LM}\label{prop:pseudo-arccharseq}
	Let $f:I\to I$ and let a nondegenerate continuum $H\subset \hat I_f$ be the inverse limit of an inverse sequence $\left(I_{k}, f_{k}\right)_{k=1}^{\infty}$, where $I_{k}\subset I$ for $k=1,2, \ldots$ are closed intervals. Then, $H$ is a pseudo-arc if and only if for every $\delta>0$ and every positive integer $k$ there is an integer $k'>k$ such that $f_{k',k}$ is $\delta$-crooked (where $f_{k',k}$ is a continuous surjection from $I_k'$ to $I_k$).
\end{proposition}

\begin{definition}\label{def:Ii}
			{\upshape
Let $f:I\to I$ be an under diagonal map and let $\eta\in (0,1)$. For every $i\geq 2$ denote by $I_i(\eta)$ a connected component of $f^{-i+1}([\eta,1])$ such that  
 the inverse limit  $\hat{I}(\eta)\subset \hat I_{f}$
$$
I_1(\eta)\xleftarrow{f|_{I_{2}(\eta)}} I_{2}(\eta) \xleftarrow{f|_{I_{3}(\eta)}} I_{3}(\eta)\xleftarrow{f|_{I_{4}(\eta)}}\ldots
$$ is a subcontinuum of $\hat{I}_f$.}
\end{definition}

Since in the rest of the paper it does not matter which inverse limit sequence $\hat I(\eta)$ we choose, by abuse of notation, we speak about general $\hat I(\eta)$ but have in mind that this $\hat I(\eta)$ is just one chosen  inverse limit sequence from possibly several choices.

\begin{observation}\label{obs:1and0}
			{\upshape
	By Observation~\ref{obs:subcontinua}, the space $\hat I(\eta)$ is a non-degenerate continuum.	
	Since $f$ is under diagonal it follows $f^{-1}(1)=\{1\}$. 
	 Thus every interval $I_i(\eta)$ from Definition~\ref{def:Ii} contains point $1$; 
	 therefore, we can compare the intervals $I_i(\eta)$ using the inclusion relation. } 
\end{observation}	

For an under diagonal map $f$  denote a fixed point of $\hat f$ by $\mathbf{1}:=(1,1,1,\ldots)\in \hat{I}_f$.
The {\em composant} of a point $x\in \hat I_f$ is the union of all proper subcontinua of $\hat I_f$ that contain point $x$. Denote the composant of $\mathbf{1}$ by $\mathcal{C}_{\mathbf 1}$. For the interpretation of the main results of this paper we note that Observation~\ref{obs:subcontinua} and Observation~\ref{obs:1and0}  yield that the composant $\mathcal{C}_{\mathbf 1}$ is the union of subcontinua $\hat I(\eta)$ for $\eta\in (0,1)$.

The following two lemmas will be crucial ingredients in the proof of the main theorems.

\begin{lemma}\label{lem:strictlyobedient}
Let $f$ be a $0$-attracted interval map. Space $\hat I_{f}$ is homeomorphic to the pseudo-arc if and only if there exists $\eta\in (0,1)$ such that $\hat I(\eta)$ is the pseudo-arc. 
\end{lemma}

\begin{proof}
$(\Rightarrow)$ The only if direction is obvious due to Observation~\ref{obs:subcontinua} since every proper subcontinuum of the pseudo-arc is the pseudo-arc.\\
$(\Leftarrow)$ For the if direction let us fix $\eta>0$ for which $\hat I(\eta)$ is the pseudo-arc. First let us assume that for every $x\in (0,1)$, $f(x)\neq 0$. Since  $f$ is $0$-attracted

\begin{equation}\label{eq:1}
x \text{ is } f \text{-attracted to } 0
\end{equation}
for every $x\in (0,1)$ and since $\hat I(\eta)$ is the pseudo-arc it follows that for every $m\geq 1$ the inverse limit
$$
f^{m}(I_{1}(\eta))\xleftarrow{f|_{f^{m-1}(I_1(\eta))}} f^{m-1}(I_{1}(\eta))\xleftarrow{f|_{f^{m-2}(I_1(\eta))}}\ldots \xleftarrow{f|_{I_1(\eta)}} I_{1}(\eta)\xleftarrow{f|_{I_2(\eta)}}\ldots 
$$
is the pseudo-arc as well since we can apply homeomorphism $\hat f^m|_{\hat I(\eta)}$ to $\hat I(\eta)$. 
But then using \eqref{eq:1} and the fact that $f$ is $0$-attracted it follows 

\begin{equation}\label{eq:2}
\exists \text{ a sequence } (\eta_i)^{\infty}_{i=1}\subset [0,1) \text{ so that } \hat I(\eta_i) \text{ is the pseudo-arc and } \eta_i\xrightarrow{i\to \infty} 0.
\end{equation}

By $\Gamma(f)$ we denote the graph of an interval map $f$ and by $d_H$ the Hausdorff distance between sets.

Now fix $\delta>0$. 
 By \eqref{eq:2} there exist integers $N,m(N)\geq 1$ such that the following three properties hold:
\begin{enumerate}
\item[(a)] $\eta_N<\frac{\delta}{12}$
\item[(b)] 	$d_H(\Gamma(f^{m(N)}),\Gamma(f^{m(N)}|_{I_1(\eta_N)})))<\frac{\delta}{24}$
\item[(c)] $f^{m(N)}(I_1(\eta_N))$ is $\frac{\delta}{3}$-crooked.
\end{enumerate}

 Items (b) and (c) from the above properties need additional explanation. Interval $I_1(\eta_N)$ is connected and contains point $1$ by the definition. Since $f$ is continuous and $0$-attracted $f^{m(N)}(I_1(\eta_N))=[t,1]$ for some $t\in (0,1)$ and thus the distance between the graphs of maps $f^{m(N)}|_{I_1(\eta_N)}$ and $f^{m(N)}$ can be chosen to be arbitrary small. Note also that in (c) we use Proposition~\ref{prop:pseudo-arccharseq}.

Now let us see that $f^{m(N)}$ is $\delta$-crooked. Let us fix $a,b\in I$ and $c<d\in I$ such that $f^{m(N)}(c)=a$ and $f^{m(N)}(d)=b$. 
The idea is that because of (a) and (b) it follows that $f^{m(N)}$ and $f^{m(N)}|_{I_1(\eta_N)}$ are close enough so that we can use (c) and applying the $\frac{\delta}{3}$-crookedness of $f^{m(N)}|_{I_1(\eta_N)}$  conclude crookedness of $f$. Let us make this more precise. 
There exist $\tilde{a},\tilde{b}\in f^{m(N)}(I_1(\eta_N)) $ and $\tilde c,\tilde d\in I_1(\eta_N)$ such that by (b) 

$$|\tilde{a}-a|<\frac{\delta}{24} \text{ and } |\tilde{b}-b|<\frac{\delta}{24},$$ 
and by (a) 
$$|c-\tilde{c}|<\frac{\delta}{12} \text{ where } c<\tilde c \text{ and } |d-\tilde{d}|<\frac{\delta}{12} \text{ where } \tilde d<d$$ 

are such that 

$$f^{m(N)}|_{I_1(\eta_N)}(\tilde{c})=\tilde b  \text{ and }  f^{m(N)}|_{I_1(\eta_N)}(\tilde{d})=\tilde a.$$ 

By (c) and (b) there exist  $\tilde c<\tilde c'<\tilde d'<\tilde d$ such that  

$$|f^{m(N)}|_{I_1(\eta_N)}(\tilde c')-f^{m(N)}(\tilde c')|<\frac{\delta}{24} \text{ and } |f^{m(N)}|_{I_1(\eta_N)}(\tilde d')-f^{m(N)}(\tilde{d'})|<\frac{\delta}{24}$$

and
 $$|f^{m(N)}|_{I_1(\eta_N)}(\tilde c)-f^{m(N)}|_{I_1(\eta_N)}(\tilde d')|<\frac{\delta}{3} \text{ and } |f^{m(N)}|_{I_1(\eta_N)}(\tilde d)-f^{m(N)}|_{I_1(\eta_N)}(\tilde c')|<\frac{\delta}{3}.$$ 
 
 Since $c<\tilde c$ and  $\tilde d<d$ it follows $c< \tilde{c'}<\tilde{d'}<d$. But then applying the above inequalities 

\begin{equation}
\begin{split}
 |f^{m(N)}(\tilde{c'})-f^{m(N)}(d)|=|f^{m(N)}(\tilde{c'})-f^{m(N)}(d)+f^{m(N)}|_{I_1(\eta_N)}(\tilde d)\\-f^{m(N)}|_{I_1(\eta_N)}(\tilde c')
 -f^{m(N)}|_{I_1(\eta_N)}(\tilde d)+f^{m(N)}|_{I_1(\eta_N)}(\tilde c')|\leq |f^{m(N)}(\tilde{c'})-\\f^{m(N)}|_{I_1(\eta_N)}(\tilde c')|+|a-\tilde{a}|
 +|f^{m(N)}|_{I_1(\eta_N)}(\tilde d)-f^{m(N)}|_{I_1(\eta_N)}(\tilde c')|<\frac{\delta}{6}+\frac{\delta}{12}+\frac{\delta}{3} <\delta 
  \end{split}
  \end{equation}
 
 and similarly 
 
 $$|f^{m(N)}(\tilde{d'})-f^{m(N)}(c)|<\delta.$$
 This shows that $f$ is crooked.\\
Note that the above argument also takes into account points $a,b$ that are possibly in $ [0,\eta_N]$. Indeed, since $f$ is under diagonal also all images of $x\in[0,\eta_N] $ will be smaller than $\eta_N$ and thus this does not have effect on the crookedness of $f$.
 
 Now assume that there exist $z\in (0,1)$ such that $f(z)=0$ and let $z$ be the largest in $(0,1)$ for which this holds. Following the preceding proof we obtain that for any sequence $(\eta_i)$ in (2) there exists $M>0$ such that $\eta_M< z$.
  If $N(m)>M$ it is thus enough to require that $f^{M}(I_1(\eta_M))$ is $\delta$-crooked (compare with (a)-(c) from the proof above) which directly implies that $f$ is crooked as well.
\end{proof}	

\begin{lemma}\label{lem:main}
Let $f$ be an under diagonal map and $\mathrm{Fix}(f)$ nowhere dense in $I$. The space $\hat I_f$ is the pseudo-arc if and only if there exists $\eta\in (0,1)$ so that $\eta$ is $f$-attracted to $0$ and $\hat I(\eta)$ is the pseudo-arc. 
 \end{lemma}

\begin{proof}
$(\Leftarrow)$ This direction follows analogously as in the proof of Lemma~\ref{lem:strictlyobedient}. The only difference from   Lemma~\ref{lem:strictlyobedient} is that we take in \eqref{eq:2} $f^i(\eta)$ for the sequence $(\eta_i)^{\infty}_{i=1}$.\\
$(\Rightarrow)$
Assume that for every $\eta\in (0,1)$
\begin{enumerate}
 \item\label{item:1} $f^{i}(\eta)\not\to 0$ as $i\to \infty$ or 
 \item\label{item:2} $\hat I(\eta)$ is not the pseudo-arc.
 \end{enumerate}
If \eqref{item:2} holds we get by Observation~\ref{obs:1and0} $\hat I(\eta)$ is non-degenerate and thus $\hat I_f$ is not the pseudo-arc.\\ 
Thus we can assume that \eqref{item:1} holds and \eqref{item:2} does not hold. Since $f$ is under diagonal and for every $x\in I\setminus \Fix(f)$, $f^{i}(x)\downarrow y\in \Fix(f)$ (actually every sequence $f^{i}(x)$ is decreasing because map $f$ is under diagonal (even strictly decreasing if $\orb(x,f)\cap \Fix(f)=\emptyset$)) there exists $p(\eta)\in \Fix(f)$ such that $0\neq p(\eta)=\lim_{i\to \infty}f^i(x)$. Indeed, $p(\eta)\in  \Fix(f)$ since otherwise $p(\eta)\neq \lim_{i\to \infty}f^i(x)$. 
Now take
 \begin{equation}
 \begin{split}
 \xi:=\inf_{\eta\in (0,1)}\{p(\eta)\in I; \text{ } \hat{I}(\eta) \text{ is the pseudo-arc} \}\newline=\\=\min_{\eta\in (0,1)}\{p(\eta)\in I; \text{ } \hat{I}(\eta) \text{ is the pseudo-arc} \}.
 \end{split}
 \end{equation}
 The last equality holds since $\Fix(f)$ is closed.
 Since \eqref{item:1} holds we obtain $\xi>0$. Also $f(\xi)=\xi$ since otherwise $\xi$ would not be the minimum as $f$ is under diagonal. Thus, summing up the above and since $f$ is an under diagonal map we obtain 
 $$\hat I_f=\underleftarrow{\lim}([0,\xi],f|_{[0,\xi]})\cup \underleftarrow{\lim}([\xi,1],f|_{[\xi,1]}).$$
 Therefore, $\hat I_f$ is decomposable and thus not the pseudo-arc. 
\end{proof}

\begin{theorem}\label{thm:nwdSetOfFP}
	Let $S$ be a non-degenerate closed nowhere dense set in $I$ such that $0,1\notin S$. There exists a crooked under diagonal interval map $f$ with $\mathrm{Fix}(f)=S\cup\{0,1\}$.
\end{theorem}

\begin{proof}
	We define $f$ in the following way. First let $f|_{[1/2,1]}=f_H|_{[1/2,1]}$, thus $f(1/2)=1/4$ and let $f(3/8)=0$, $f(1/4)=1/4$ and $f|_{[1/4,1/2]}$ be a piecewise linear map with two pieces of monotonicity on the interval $[3/8,1/2]$. Now, let $S\subset (0,1/4)$ and let $f|_{[0,1/4]}$ be chosen so that $\Fix(f|_{(0,1/4]})=S$ and for $x\in (0,1/4]\setminus S$ let $f(x)<x$. Since $f^2(\frac{\sqrt{3}}{2\sqrt{2}})=f(3/8)=0$ we take $\eta=\frac{\sqrt{3}}{2\sqrt{2}}$ and we apply Lemma~\ref{lem:main}, which finishes the proof.
\end{proof}

Since we have uncountably many non-homeomorphic closed nowhere dense sets in $I$ we obtain the following.

\begin{corollary}\label{cor:non-conjobedient}
There exist uncountably many topologically non-conjugated under diagonal crooked  maps with pairwise different sets of fixed points.
\end{corollary}

We will only implicitly use the well known (standard) definition of topological entropy for compact metric spaces from \cite{AKM,Bowen} and thus we will not define it here.

Note that under diagonal and above diagonal interval maps have topological entropy $0$.

Therefore, Corollary~\ref{cor:non-conjobedient} and Theorem~\ref{thm:non-conjstrobedient} imply the following two statements.

\begin{corollary}
	There exists an uncountable set of crooked interval maps $\mathcal{T}\subset C(I)$ with topological entropy $0$ such that any two different $f, g\in \mathcal{T}$ are not topologically conjugate.
\end{corollary}

A question that remains to be answered is whether there are uncountably many pairwise non-conjugated maps that have the same structure of fixed points as the Henderson's map $f_H$, i.e. $0$-attracted maps.

\begin{definition} \label{def:strictlymonotone}
	{\upshape
	Let $f\in C(I)$. We say that $f$ is {\em countably piecewise strictly monotone}, if there are infinite families
$\mathcal{I}_{f}=\left\{\left[p_{k}, q_{k}\right] \mid k\in \N \right\}$
and $\mathcal{D}_{f}=\left\{\left[r_{k}, s_{k}\right] \mid k\in \N \right\}$
of closed intervals in $I$ such that

\begin{enumerate}
	\item for all positive integers $k$ and $\ell$,
\end{enumerate}
$$
k \neq \ell \Longrightarrow\left[p_{k}, q_{k}\right] \cap\left[p_{\ell}, q_{\ell}\right]=\emptyset \text { and }\left[r_{k}, s_{k}\right] \cap\left[r_{\ell}, s_{\ell}\right]=\emptyset
$$
\begin{enumerate}
	\setcounter{enumi}{1}
	\item for all positive integers $k$ and $\ell$,
\end{enumerate}
$$
\left(p_{k}, q_{k}\right) \cap\left(r_{\ell}, s_{\ell}\right)=\emptyset
$$
\begin{enumerate}
	\setcounter{enumi}{2}
	\item the closure $\mathrm{Cl}\left(\left(\bigcup_{k=1}^{\infty}\left[p_{k}, q_{k}\right]\right) \cup\left(\bigcup_{k=1}^{\infty}\left[r_{k}, s_{k}\right]\right)\right)=I$,
	
	\item for each positive integer $k$, the map $f$ is strictly increasing on $\left[p_{k}, q_{k}\right]$, and
	
	\item for each positive integer $k$, the map $f$ is strictly decreasing on $\left[r_{k}, s_{k}\right]$.
	
\end{enumerate}}
\end{definition}

Note that the families $\mathcal{I}_{f}$ and $\mathcal{D}_{f}$ from Definition~\ref{def:strictlymonotone} are uniquely determined for any countably piecewise strictly monotone function $f:I \to I$.

\begin{definition} 
			{\upshape
	Let $f\in C(I)$ be a countably piecewise strictly monotone map and let
	$\mathcal{D}_{f}=\left\{\left[r_{k}, s_{k}\right] \mid k\in \N\right\},$ i.e., for all positive integers $k$ and $\ell$, if $
	k \neq \ell$ then $\left[r_{k}, s_{k}\right] \cap\left[r_{\ell}, s_{\ell}\right]=\emptyset$.
	We define $\mathcal{A}(f)$ to be the set of all points $t \in[0,1]$ for which there is a convergent subsequence $\left(\left[r_{i_{k}}, s_{i_{k}}\right]\right)$ of the sequence $\left(\left[r_{k}, s_{k}\right]\right)$ in the Hausdorff metric, such that
	$$
	\lim _{k \rightarrow \infty}\left[r_{i_{k}}, s_{i_{k}}\right]=\{t\}.
	$$ We call points in $\mathcal{A}(f)$  the {\em accumulation points of $f$}.}
\end{definition}

\begin{observation}\label{obs:accpts}
			{\upshape
	Let $f,g\in C(I)$ be two countably piecewise strictly monotone functions. If $f$ and $g$ are topologically conjugate and $\varphi$ : $I \rightarrow I$ is a homeomorphism such that $f=\varphi^{-1} \circ g \circ \varphi$, then

$$
\mathcal{A}(g)=\{\varphi(t) \mid t \in \mathcal{A}(f)\}.
$$ In particular, $|\mathcal{A}(g)|=|\mathcal{A}(f)|$.}
\end{observation}

\begin{theorem}\label{thm:non-conjstrobedient}
	There exists an uncountable set of crooked $0$-attracted interval maps $\mathcal{T}\subset C(I)$ such that any two different $f, g\in \mathcal{T}$ are not topologically conjugate.
\end{theorem}

\begin{proof}
	We will show that there are uncountably many non-conjugate maps such that $f|_{[0,1/2]}$ with $f(0)=0$, $f(1/2)=1/4$ and $f(x)<x$ for $x\in (0,1/2)$ and $f|_{[1/2,1]}=f_H|_{[1/2,1]}$. By Lemma~\ref{lem:strictlyobedient} such $f$ is crooked so we have complete freedom to choose $f|_{[0,1/2]}$ as long as $f(x)<x$ for all $x\in (0,1/2]$. At this point we could finish the proof by saying that we have uncountably many options to choose the Cantor-Bendixson rank of the set of strict critical points of the map $f|_{[0,1/2]}$.\\
	Let us sketch the arguments. Let $f|_{[0,1/2]}$ have $n$ critical points. We have a choice to make each strict critical point an accumulation point or not and since conjugacy of a $0$-attracted map needs to preserve $0$ it also needs to preserve ordering of accumulation points. Accumulation points are preserved under the conjugacy by Observation~\ref{obs:accpts}. In such a way we get countably (varying the number of strict critical points) over countably (by adding accumulation points) many pairwise non-conjugate $0$-attracted choices to construct map $f$. 
\end{proof}

\begin{corollary}
	There exists uncountably many pairwise non-conjugated crooked maps with topological entropy $0$ and exactly two fixed points. 
\end{corollary}

We could perform a construction similar to the one in Theorem~\ref{thm:non-conjstrobedient} for any under diagonal map with a nowhere dense set of fixed points and obtain the following theorem. 

\begin{theorem}
Let $S$ be a non-degenerate closed nowhere dense set in $I$ such that $0,1\notin S$. There exist uncoutably many pairwise non-conjugate under diagonal crooked maps $f$ with $\mathrm{Fix}(f)=S\cup\{0,1\}$.
\end{theorem}

\begin{definition}
			{\upshape
	Let $f:I\to I$ be an above diagonal map and let $\eta\in (0,1)$.
	Define by $I'_i(\eta)$ a connected component of $f^{-i}([0,\eta])$. Denote a fixed point of $\hat f$ by $\mathbf{0}:=(0,0,0,\ldots)\in \hat{I}_f$ and the composant of $\mathbf{0}$ by $\mathcal{C}_{\mathbf 0}$. }
\end{definition}

\begin{remark}\label{rem:reverseobedient}
    {\upshape
	Let $f$ be an above diagonal interval map. Note that $\varphi:I\to I$ defined by $\varphi(x)=1-x$ for every $x\in I$ conjugates $f$ to a map $g\in C_{ud}(I)$, i.e., $\varphi^{-1}\circ f \circ \varphi\in C_{ud}(I)$. Thus, $\varphi$ is actually a homeomorphism of $C_{ud}(I)$ to $C_{ad}(I)$. Therefore, the results from the previous section follow analogously for the set $C_{ad}(I)$ if we replace $\hat I(\eta)$ with similarly defined $\hat I'(\eta)$ and the role of ${\bf 1}$ and $\mathcal{C}_{\bf 1}$ by ${\bf 0}$ and $\mathcal{C}_{\bf 0}$ respectively.}
\end{remark}

Thus we have the following theorem.

\begin{theorem}\label{thm:charobedientandreverseobedient}
	Map $f$ is crooked and under diagonal or above diagonal interval map if and only if $\mathrm{Fix}(f)$ is nowhere dense in $I$ and 
	\begin{enumerate}
		\item either there exists $\eta\in (0,1)$ such that $\hat I(\eta)$ is the pseudo-arc and $\eta$ is $f$-attracted to $0$ or 
		\item there exists $\eta\in (0,1)$ such that $\hat I'(\eta)$ is the pseudo-arc and $\eta$ is $f$-attracted to $1$.
	\end{enumerate}
\end{theorem}

\begin{proof}
	The proof of this theorem follows from Lemma~\ref{lem:main} and Remark~\ref{rem:reverseobedient} if we prove that provided $\mathrm{Fix}(f)$ is dense somewhere in $I$, then $\hat{I}_f$ is not the pseudo-arc.
	Say $\mathrm{Fix}(f)$ is dense somewhere in $I$, denote that set by $A\subset I$. Since $f$ is continuous, the closure $\mathrm{Cl}(A)\subset \Fix(f)$. Thus there exists an open interval $U\subset A$.  Fix a non-degenerate closed interval $J\subset U$. Inverse limit $\underleftarrow{\lim}(J,f|_{J})\subset \hat{I}_f$ is an arc which finishes the proof.
\end{proof}

\section{Acknowledgments}
I would like to thank Ana Anu\v si\' c, Iztok Bani\v c, Jan Boro\'nski and Piotr Oprocha for remarks that helped improve the exposition of the paper.
J. \v Cin\v c was partially supported by the Slovenian research agency ARRS grant J1-4632 and by European Union's Horizon Europe Marie Sk\l odowska-Curie grant HE-MSCA-PF- PFSAIL - 101063512.

\vspace{0.7cm}

\begin{table}[ht]
\begin{tabular}[t]{p{1.5cm}  p{10.5cm} }
\includegraphics [width=.09\textwidth]{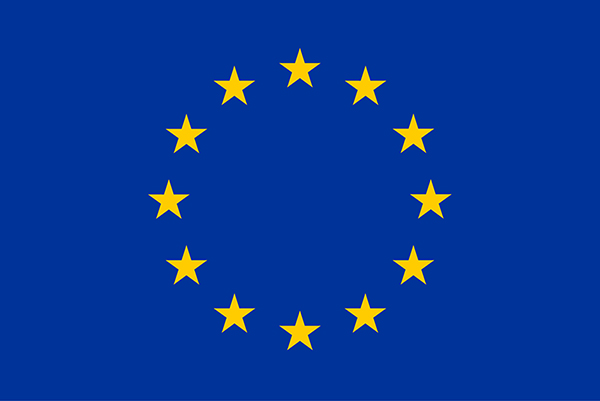} &
\vspace{-1cm}
This research is part of a project that has received funding from the European Union's Horizon Europe research and innovation programme under the Marie Sk\l odowska-Curie grant agreement No 101063512.\\
\end{tabular}
\end{table}

\end{document}